\documentclass[12pt,twoside]{amsart}
\usepackage[all]{xy}
\usepackage{graphicx}

\usepackage{xcolor}

\usepackage{amssymb}

\title[An affine local criterion for toric projective space bundles]
{An affine local criterion for toric projective space bundles} 
\author[Osamu Fujino]{Osamu Fujino}
\author[Hiroshi Sato]{Hiroshi Sato}

\subjclass[2020]{Primary 14M25; Secondary 14E30.}
\date{2026/8/11, version 0.02}
\keywords{Toric Mori theory, lengths of 
extremal rays, Fano contractions, projective space bundles, 
non-$\mathbb{Q}$-factorial varieties, 
relative canonical divisors}
\address{Department of 
Mathematics, Graduate School of Science, 
Kyoto University, Kyoto 606-8502, Japan}
\email{fujino@math.kyoto-u.ac.jp}
\address{Department of Applied Mathematics, 
Faculty of Sciences, Fukuoka University, 
8-19-1, Nanakuma, Jonan-ku, Fukuoka 814-0180, Japan}
\email{hirosato@fukuoka-u.ac.jp}

\makeatletter
    
    \@addtoreset{equation}{section}
\makeatother

\newcommand{\NE}[0]{{\operatorname{NE}}}

\newcommand{\RR}{\mathbb R}
\newcommand{\ZZ}{\mathbb Z}
\newcommand{\PP}{\mathbb P}
\newcommand{\GGm}{\mathbb G_m}
\newcommand{\Cone}{\operatorname{Cone}}
\newcommand{\Span}{\operatorname{Span}}

\newtheorem{thm}{Theorem}[section]
\newtheorem{lem}[thm]{Lemma}
\newtheorem{cor}[thm]{Corollary}
\newtheorem{prop}[thm]{Proposition}

\theoremstyle{definition}
\newtheorem{ex}[thm]{Example}

\newtheorem{rem}[thm]{Remark}
\newtheorem*{ack}{Acknowledgments}   
\newtheorem{que}[thm]{Question}

\begin{document}

\begin{abstract}
We study when an equidimensional toric morphism 
is forced to be a
projective-space bundle.  Our main result is an affine rigidity
theorem: if the base space is 
affine, the toric relative canonical divisor is
\(\mathbb Q\)-Cartier, and its negative has degree greater than the
relative dimension on every complete curve, then the morphism is
equivariantly a trivial projective-space bundle. 
As an 
application, we derive a projective-space-bundle theorem for
equidimensional toric contractions associated to long extremal rays,
without assuming \(\mathbb Q\)-factoriality.
\end{abstract}

\maketitle

\section{Introduction} 

Projective-space-bundle structures arise naturally in the study of
toric extremal contractions associated to long extremal rays; 
see \cite{fujinosato-notes2}, \cite{fujino-length}, 
\cite{foliation1}, and \cite{foliation2}. 
The authors proved that if \(X\) is a \(\mathbb Q\)-factorial
projective toric variety and
\[
  \varphi_R\colon X\longrightarrow W
\]
is a Fano contraction associated to a \(K_X\)-negative extremal ray
\(R\), then the inequality
\[
  l(R)>\dim X-\dim W
\]
implies that \(\varphi_R\) is a projective-space bundle; see
\cite[Corollary~3.3]{fujino-length}. 
The purpose of this paper is to establish an affine local version
which does not require the total space to be \(\mathbb Q\)-factorial.

Let
\[
  \varphi\colon X\longrightarrow Y
\]
be a proper surjective equidimensional toric morphism with connected
fibers.  We write
\[
  d:=\dim X-\dim Y.
\]
For a torus-invariant Weil divisor \(D\) on \(Y\), we denote by
\[
  \varphi^{[*]}D
\]
the torus-invariant Weil divisor on \(X\) obtained by pulling back
\(D\) over the smooth locus of \(Y\) and then taking its closure in
\(X\).  In particular, we define the toric relative canonical divisor
by
\[
  K_{X/Y}:=K_X-\varphi^{[*]}K_Y.
\]
If \(K_Y\) is \(\mathbb Q\)-Cartier, then
\(\varphi^{[*]}K_Y=\varphi^*K_Y\), and this agrees with the usual
relative canonical divisor.

Our main result is an affine criterion for \(\varphi\) to be a trivial
projective-space bundle.  More precisely, Theorem~\ref{thm:affine-local}
states that if \(Y\) is affine, \(K_{X/Y}\) is \(\mathbb Q\)-Cartier,
and
\[
  -K_{X/Y}\cdot C>d
\]
for every complete integral curve \(C\subset X\), then there is an
equivariant isomorphism over \(Y\)
\[
  X\simeq\mathbb P^d\times Y.
\]
Thus, although neither \(X\) nor \(Y\) is assumed to be
\(\mathbb Q\)-factorial, the relative numerical condition forces the
morphism to have the simplest possible toric structure.

We briefly explain the main steps of the proof.  First, the numerical
assumption implies that the general fiber of \(\varphi\) is
\(\mathbb P^d\).  Hence the fan in the kernel of the corresponding
lattice homomorphism is the standard fan of \(\mathbb P^d\).

Second, we study the inverse image of the toric chart associated to a
ray of the fan of \(Y\).  After separating a torus factor, the resulting
morphism is a toric morphism to \(\mathbb A^1\).  A relative version of
the long-ray argument of the authors shows that each ray of the fan
of \(Y\) has exactly one lift to the fan of \(X\), and that a primitive
generator is mapped to the primitive generator of the corresponding
ray of the base.

Third, equidimensionality and properness force these lifted rays to
form a horizontal subfan.  Moreover, at the level of real cones, the
fan of \(X\) is the sum of this horizontal subfan and the standard fan
of \(\mathbb P^d\).  In other words, the fan is weakly split.

It remains to rule out a finite lattice-index defect.  If the
horizontal lattice were not saturated, then one of the standard walls
of the fiber fan would determine a complete torus-invariant curve
\(C\subset X\) such that
\[
  -K_{X/Y}\cdot C
  =
  \frac{d+1}{q}
  \leq d
\]
for some integer \(q\geq2\).  This contradicts the numerical
assumption.  The fan is therefore split in the integral sense, and the
standard split-fan criterion gives
\[
  X\simeq\mathbb P^d\times Y
\]
over the affine toric variety \(Y\).

As an application of the affine theorem, we obtain a global
projective-space-bundle statement for equidimensional toric
contractions satisfying the corresponding relative length condition.
The relative Picard number is not needed in the affine theorem itself.
In the global extremal setting, it is used to ensure that every
complete curve contained in the inverse image of an affine toric chart
belongs to the same contracted extremal ray and hence satisfies the
required numerical inequality.

Finally, we give an example showing that equidimensionality is
essential.  Without this assumption, a torus-invariant divisor may be
contained in a special fiber, and the dimension of that fiber may be
larger than the dimension of the general fiber, even when the relevant
relative length is larger than the relative dimension.

\begin{ack}
The first author was partially supported by JSPS KAKENHI Grant 
Number JP23K20787. The second author was partially supported by JSPS KAKENHI Grant Number 24K06679. 

The authors acknowledge the use of Rethlas and ChatGPT during the development of this work. These tools were used in the exploratory phase of the research, including discussions of formulations and the construction of examples. All mathematical arguments, proofs, and final verifications were carried out independently by the authors, who assume full responsibility for the contents of this paper.
\end{ack}

\section{Preliminaries}

We work over an algebraically closed field. 
We use the standard notation and 
terminology of toric geometry and toric Mori 
theory. For details, see, for 
example, \cite{Oda}, \cite{fulton}, 
\cite{cox-little-toric} and \cite[Chapter 14]{matsuki}. 
Throughout the paper, a curve means an integral curve. 

\smallskip

We first recall some known results. Lemma \ref{ext:fujino_toric_cone} and Lemma \ref{f-thm3.2.1} extract the necessary parts from \cite[Theorem 0.1]{fujino-notes} and \cite[Theorem 3.2.1]{fujinosato-notes2}, respectively.

\begin{lem}[{\cite[Theorem 0.1]{fujino-notes}}]
\label{ext:fujino_toric_cone}
Let \(X\) be a \(\mathbb{Q}\)-Gorenstein projective toric variety of dimension \(n\), and let \(R\subset \NE(X)\) be a \(K_X\)-negative extremal ray. Then
\[
l(R)\leq n
\]
unless \(X\simeq \mathbb{P}^n\), where
\[
l(R):=
\min\left\{
-K_X\cdot C
\, \left|\,
C\text{ is a complete integral curve and }[C]\in R
\right.\right\}
\]
is the length of \(R\).
\end{lem}

\begin{lem}[{\cite[Theorem 3.2.1]{fujinosato-notes2}}]
\label{f-thm3.2.1}
Let $f\colon X\to Y$ be a projective toric morphism with $\dim X=n$. 
Assume that $K_X$ is $\mathbb Q$-Cartier. 
Let $R$ be a $K_X$-negative 
extremal ray of $\NE(X/Y)$ and let $\varphi_R\colon X\to W$ 
be the contraction morphism associated to $R$. 
Assume that $\varphi_R$ is birational. 
Then we obtain 
\[
l(R)<e+1, 
\]
where \[e=\max_{w\in W} \dim \varphi^{-1}_R(w)\leq n-1. \] 
When $e=n-1$, we have a sharper inequality 
\[
l(R)\leq e=n-1. 
\]
\end{lem}

The following characterization of equidimensionality is well known to experts.

\begin{prop}[Equidimensionality and images of cones]
\label{prop:AK}
Let $\varphi_p\colon X_\Sigma\to Y_\Delta$ be a surjective toric morphism and let $p\colon N\to N_Y$ be the lattice homomorphism defining $\varphi_p$. Then $\varphi_p$ is
equidimensional if and only if $p_{\RR}(\sigma)$ is a cone of $\Delta$ for every
$\sigma\in\Sigma$.
\end{prop}

\begin{proof}
This is the toric case of Abramovich--Karu's characterization of equidimensional
toroidal morphisms; see~\cite[Lemma~4.1]{AK}.
\end{proof}

\section{The case of a one-dimensional affine base}

The following theorem can be regarded as an affine-line analogue of 
\cite[Theorem 3.2.9]{fujinosato-notes2}. 
It is the main ingredient in the proof of 
the affine local criterion established in Section \ref{mainsection}, 
but also seems to be of independent interest. 

\begin{thm}[{The affine-line version of \cite[Theorem 3.2.9]{fujinosato-notes2}}]
\label{thm:A1-version}
Let
\(
  f\colon X\to Y
\)
be a projective surjective toric morphism with connected fibers.
Assume that $Y\simeq \mathbb A^1$, $\dim X=n\geq 2$,
and that $X$ is \(\mathbb Q\)-Gorenstein. Suppose moreover that
\[
  -K_X\cdot C>n-1
\]
for every complete integral curve \(C\subset X\).

Then there exists an equivariant isomorphism over \(Y\)
\[
  X\simeq \mathbb P^{n-1}\times Y
\]
under which \(f\) is identified with the second projection.
\end{thm}

\begin{proof} 
We divide the proof into four steps. 

\medskip
\noindent\textit{Step 1. Choosing an extremal ray. } We first choose a suitable extremal ray and show that its contraction is of fiber type. 

Since \(n\geq 2\), a fiber of \(f\) contains a complete curve. By
assumption,
\[
  K_X\cdot C<0
\]
for every complete curve \(C\subset X\). In particular, \(K_X\) is
not \(f\)-nef.

Since \(f\) is a projective toric morphism, the relative Kleiman--Mori cone
$\NE(X/Y)$ is a rational polyhedral cone. Hence it contains a
\(K_X\)-negative extremal ray \(R\). Let
$
  \varphi_R\colon X\to W
$ be the contraction of \(R\). Since \(\varphi_R\) is a contraction
over \(Y\), there is a projective toric morphism
$h\colon W\to Y$ such that
$f=h\circ\varphi_R$. 

Every curve whose numerical class belongs to \(R\) is complete.
Therefore the hypothesis gives
\begin{equation}\label{eq:A1-length}
  l(R)>n-1.
\end{equation}

We claim that \(\varphi_R\) is a Fano contraction, that is, $\dim W<\dim X$. Suppose that
\(\varphi_R\) is birational, and put
\[
  e:=\max_{w\in W}\dim\varphi_R^{-1}(w).
\]
By Lemma \ref{f-thm3.2.1}, if \(e\leq n-2\), then
\[
  l(R)<e+1\leq n-1,
\]
contrary to \eqref{eq:A1-length}. If \(e=n-1\), the sharper part of the same theorem
gives
\[
  l(R)\leq n-1,
\]
again contradicting \eqref{eq:A1-length}. Thus \(\varphi_R\) is not birational.
Consequently, \(\varphi_R\) is a Fano contraction.

\medskip
\noindent\textit{Step 2. Identifying the contraction. } 
We next show that the extremal contraction obtained above coincides with the original morphism. 

Let \(F\) be a general fiber of \(\varphi_R\), and put
\[
  r:=\dim F=\dim X-\dim W. 
  \]
Since \(h\colon W\to Y\) is surjective and \(\dim Y=1\), we have
$\dim W\geq 1$, 
and hence
$
  r\leq n-1$. 

The fiber \(F\) is a projective toric \(r\)-fold. Over the big torus
of \(W\), adjunction gives
\[
  K_F=K_X|_F.
\]
Thus, for every integral curve \(C\subset F\),
\begin{equation}\label{eq:A1-fiber-inequality}
  -K_F\cdot C=-K_X\cdot C>n-1\geq r.
\end{equation}

By Lemma~\ref{ext:fujino_toric_cone}, if
\(F\not\simeq\mathbb P^r\), then there exists an
integral curve \(C\subset F\) such that
\[
  -K_F\cdot C\leq r,
\]
contrary to \eqref{eq:A1-fiber-inequality}. Therefore
\[
  F\simeq\mathbb P^r.
\]
For a line \(L\subset F\simeq\mathbb P^r\), we have
\[
  -K_X\cdot L=-K_F\cdot L=r+1.
\]
The hypothesis therefore gives
$
  r+1>n-1$. 
Since \(r\leq n-1\), we obtain
$
  r=n-1$. 
Consequently,
\[
  \dim W=1.
\] 
Let \(q:N_W\to N_Y\) be the lattice homomorphism defining \(h\).
Since both \(f\) and \(\varphi_R\) have connected fibers, the
corresponding lattice homomorphisms are surjective. Hence \(q\) is
surjective. Since $N_W$ and $N_Y$ both have rank one, $q$ is an isomorphism.
Hence $h$ is a projective birational morphism between normal curves,
and therefore $h$ is an isomorphism. 
After identifying \(W\) with \(Y\), we have \(f=\varphi_R\). 

In particular,
\begin{equation}\label{eq:A1-rho-one}
  \rho(X/Y)=1.
\end{equation}

\medskip
\noindent\textit{Step 3. The $\mathbb Q$-factoriality of $X$. } 
We now prove that $X$ is $\mathbb Q$-factorial. 

Suppose that \(X\) is not \(\mathbb Q\)-factorial. Let
$
  \pi\colon \widetilde X\to X
$
be a small projective toric \(\mathbb Q\)-factorialization. Since
\(\pi\) is small and \(K_X\) is \(\mathbb Q\)-Cartier,
\begin{equation}\label{eq:A1-crepant}
  K_{\widetilde X}=\pi^*K_X.
\end{equation}
Put
\[
  \widetilde f:=f\circ\pi\colon\widetilde X\to Y.
\]

Since \(\pi\) is nontrivial and projective, 
\(\rho(\widetilde X/X)>0\), and hence 
\begin{equation}\label{eq:A1-rho-tilde}
  \rho(\widetilde X/Y)>\rho(X/Y)=1.
\end{equation}

The divisor \(K_{\widetilde X}\) is not \(\widetilde f\)-nef, so we
can choose a \(K_{\widetilde X}\)-negative extremal ray
$
  \widetilde R\subset\NE(\widetilde X/Y)$. 
A curve whose class belongs to \(\widetilde R\) is not
\(\pi\)-exceptional, because by \eqref{eq:A1-crepant} the divisor \(K_{\widetilde X}\)
has degree zero on every \(\pi\)-exceptional curve.

Let \(\widetilde C\) be an integral curve with
\([\widetilde C]\in\widetilde R\), and let
$
  C:=\pi(\widetilde C)
$
with the reduced structure. If the mapping degree of
\(\widetilde C\to C\) is \(m>0\), then
\[
  -K_{\widetilde X}\cdot\widetilde C
  =
  -\pi^*K_X\cdot\widetilde C
  =
  m(-K_X\cdot C)
  >
  n-1.
\]
Thus
\begin{equation}\label{eq:A1-length-tilde}
  l(\widetilde R)>n-1.
\end{equation}

Repeating the arguments of Steps~1 and~2 for the fixed extremal ray 
\(\widetilde R\), 
using \eqref{eq:A1-length-tilde}, we see
that this contraction is a Fano contraction and that its base is
isomorphic to \(Y\). Hence it is precisely the morphism
$
  \widetilde f\colon\widetilde X\to Y$. 
Since it is the contraction of a single extremal ray, this gives
$
  \rho(\widetilde X/Y)=1$, 
contrary to \eqref{eq:A1-rho-tilde}. Therefore \(X\) is \(\mathbb Q\)-factorial.

\medskip
\noindent\textit{Step 4. Description of the fan. } 
Finally, we determine the fan explicitly and conclude the proof. 

Let
\[
  p\colon N\longrightarrow N_Y\simeq\mathbb Z
\]
be the lattice homomorphism defining \(f\). Since \(f\) has connected
fibers, \(p\) is surjective. Put
\[
  N_F:=\ker p.
\]

The general fiber is \(\mathbb P^{n-1}\). Therefore its fan in
\((N_F)_{\mathbb R}\) has primitive ray generators
\[
  v_1,\ldots,v_n
\]
which, after choosing a basis of \(N_F\), may be written as
\begin{equation}\label{eq:A1-fiber-fan}
  v_1=e_1,\quad\ldots,\quad v_{n-1}=e_{n-1},
  \qquad
  v_n=-e_1-\cdots-e_{n-1}.
\end{equation}

Since \(X\) is \(\mathbb Q\)-factorial and \(\rho(X/Y)=1\), the
standard toric divisor sequence shows that there is exactly one
nonvertical ray of the fan of \(X\). 
Let \(v_+\) be its primitive
generator. Choose \(e_n\in N\) such that
\[
  p(e_n)=1.
\]
Then
\[
  N=N_F\oplus\mathbb Ze_n,
\]
and we may write
\begin{equation}\label{eq:A1-vplus}
  v_+
  =
  b_1e_1+\cdots+b_{n-1}e_{n-1}+ae_n,
  \qquad a>0.
\end{equation}

By replacing \(e_n\) with
\[
  e_n+c_1e_1+\cdots+c_{n-1}e_{n-1}
\]
for suitable integers \(c_i\), we may assume
\begin{equation}\label{eq:A1-bi-range}
  0\leq b_i<a
  \qquad
  (1\leq i\leq n-1).
\end{equation}

Properness of \(f\) and the description of the general fiber imply
that the maximal cones of the fan of \(X\) are
\begin{equation}\label{eq:A1-maximal-cones}
  \operatorname{Cone}
  (v_+,v_1,\ldots,\widehat v_i,\ldots,v_n),
  \qquad
  1\leq i\leq n.
\end{equation}

We claim that \(b_i=0\) for every \(i\). Suppose otherwise. After
renumbering the \(v_i\)'s, we may assume that
\[
  0<b_1<a.
\]
Let
\[
  \mu
  :=
  \operatorname{Cone}
  (v_2,\ldots,v_{n-1},v_+),
\]
where the list \(v_2,\ldots,v_{n-1}\) is empty when \(n=2\), and put
\[
  C:=V(\mu).
\]
The curve \(C\) is contained in the fiber over the origin of
\(\mathbb A^1\), and hence is complete.

Let \(D_i:=V(v_i)\) and \(D_+:=V(v_+)\). The principal divisor
relations give
\[
  D_i-D_n+b_iD_+\sim 0
  \qquad
  (1\leq i\leq n-1)
\]
and
\[
  aD_+\sim 0.
\]
Consequently,
\begin{equation}\label{eq:A1-divisor-intersections}
  D_+\cdot C=0
  \quad\text{and}\quad
  D_i\cdot C=D_n\cdot C
  \qquad
  (1\leq i\leq n-1).
\end{equation}

By the standard toric intersection formula,
\begin{equation}\label{eq:A1-intersection-multiplicity}
  D_1\cdot C
  =
  \frac{
    \operatorname{mult}(\mu)
  }{
    \operatorname{mult}
    (\operatorname{Cone}(v_1,\ldots,v_{n-1},v_+))
  }
  =
  \frac{\gcd(a,b_1)}{a}.
\end{equation}
Using
\[
  -K_X=D_1+\cdots+D_n+D_+
\]
and \eqref{eq:A1-divisor-intersections}, we obtain
\begin{equation}\label{eq:A1-anticanonical-degree}
  -K_X\cdot C
  =
  n\,\frac{\gcd(a,b_1)}{a}.
\end{equation}
Since \(0<b_1<a\), the integer \(\gcd(a,b_1)\) is a proper divisor of
\(a\), and hence
\[
  \frac{\gcd(a,b_1)}{a}\leq\frac12.
\]
Therefore
\[
  -K_X\cdot C
  \leq
  \frac n2
  \leq
  n-1,
\]
contrary to the hypothesis.

Thus $
  b_1=\cdots=b_{n-1}=0$. 
It follows from \eqref{eq:A1-vplus} that $v_+=ae_n$. 
Since \(v_+\) is primitive, we must have
$a=1$. 
Therefore $v_+=e_n$. 

By \eqref{eq:A1-fiber-fan}, \eqref{eq:A1-maximal-cones}, and \(v_+=e_n\), the fan of \(X\) is the product of the
standard fan of \(\mathbb P^{n-1}\) and the fan of \(\mathbb A^1\).
Hence there is an equivariant isomorphism over \(Y\)
\[
  X\simeq\mathbb P^{n-1}\times\mathbb A^1.
\]
Under this isomorphism, \(f\) is the second projection.
\end{proof}

Although Theorem \ref{thm:A1-version} is 
sufficient for our purposes, 
the proof suggests that its numerical assumption 
may not be optimal. 
Since this theorem is the key local ingredient of 
Theorem \ref{thm:affine-local}, 
a positive answer to the following question 
would immediately strengthen both the affine criterion and 
its global application to 
long extremal toric contractions (cf. \cite{fujino-length}).

\begin{que} 
Is it possible to weaken the assumption $-K_X\cdot C>n-1$ in Theorem \ref{thm:A1-version} to require only that $-K_X\cdot C>\frac{n}{2}$ for every complete integral curve, and $-K_X\cdot C>n-1$ for every curve \(C\) contained in a general fiber? 

Note that the condition $-K_X\cdot C>n-1$ for 
every curve $C$ contained in 
a general fiber is equivalent to the general fiber being isomorphic to $\mathbb{P}^{n-1}$. If Theorem~\ref{thm:A1-version} holds under these weaker assumptions,
then in Theorem~\ref{thm:affine-local} one may replace
\eqref{eq:curve-inequality} by the condition
\[
-K_{X/Y}\cdot C>\frac{d+1}{2}
\]
for every complete integral curve \(C\subset X\), together with
\[
-K_{X/Y}\cdot C>d
\]
for every curve \(C\) contained in a general fiber. 
\end{que}

\section{Main Theorem}\label{mainsection} 

We first explain the relative canonical divisor used throughout this section.  We fix the
standard torus-invariant canonical divisors
\[
  K_X=-\sum_{\rho\in\Sigma(1)}D_\rho,
  \qquad
  K_Y=-\sum_{\alpha\in\Delta(1)}D_\alpha.
\]
Let
\[
  \varphi=\varphi_p\colon X_\Sigma\longrightarrow Y_\Delta
\]
be a proper surjective equidimensional toric morphism with connected fibers. Note that $p\colon N\to N_Y$ is the lattice homomorphism defining $\varphi$. For $\rho\in\Sigma(1)$ and $\alpha\in\Delta(1)$, we denote by
$u_\rho\in N$ and $w_\alpha\in N_Y$ the primitive lattice generators
of $\rho$ and $\alpha$, respectively. 
Since $Y$
is normal, the complement of its smooth locus $Y_{\mathrm{sm}}$ has codimension at least
two.  Equidimensionality gives
\[
  \operatorname{codim}_X\varphi^{-1}(Y\setminus Y_{\mathrm{sm}})\geq2.
\]
Hence, for a torus-invariant Weil divisor $D$ on $Y$, the Cartier pullback of
$D|_{Y_{\mathrm{sm}}}$ on $\varphi^{-1}(Y_{\mathrm{sm}})$ has a unique closure as a
Weil divisor on $X$.  We denote this closure by
\[
  \varphi^{[*]}D.
\]
If $D$ is $\mathbb Q$-Cartier, then $\varphi^{[*]}D$ agrees with the usual pullback
$\varphi^*D$.

More explicitly, write
\[
  D=\sum_{\alpha\in\Delta(1)}a_\alpha D_\alpha.
\]
By Proposition~\ref{prop:AK}, every ray of $\Sigma$ is mapped either to the origin or to
a ray of $\Delta$. If $p(u_\rho)=m_\rho w_\alpha$ with $m_\rho>0$, 
then the coefficient of $D_\rho$ in
$\varphi^{[*]}D$ is $m_\rho a_\alpha$. A vertical divisor, that is, a divisor
corresponding to a ray contained in $\ker(p)_{\mathbb R}$, has coefficient zero in
$\varphi^{[*]}D$.

We define the toric relative canonical divisor by
\[
  K_{X/Y}:=K_X-\varphi^{[*]}K_Y.
\]
When $K_Y$ is $\mathbb Q$-Cartier, this is the usual relative canonical divisor
$K_X-\varphi^*K_Y$.

\begin{thm}[Affine local projective-space-bundle criterion]
\label{thm:affine-local}
Let
\[
  \varphi\colon X\longrightarrow Y
\]
be a proper surjective toric morphism with connected fibers.  Assume that $Y$ is an
affine toric variety and that $\varphi$ is equidimensional of relative dimension
\[
  d:=\dim X-\dim Y>0.
\]
Assume moreover that the toric relative canonical divisor $K_{X/Y}$ is
$\mathbb Q$-Cartier and that
\begin{equation}
  -K_{X/Y}\cdot C>d
  \label{eq:curve-inequality}
\end{equation}
for every complete integral curve $C\subset X$.  Then there is an equivariant
isomorphism over $Y$
\[
  X\simeq \PP^d\times Y,
\]
under which $\varphi$ is the second projection.
\end{thm}

\begin{rem}\label{rem:general-divisor}
The proof of Theorem~\ref{thm:affine-local} also gives the following slightly more
general statement.  One may replace $K_{X/Y}$ by
\[
  K_X-\varphi^{[*]}D
\]
for any torus-invariant Weil divisor $D$ on $Y$, provided that this divisor is
$\mathbb Q$-Cartier and the corresponding curve inequality holds.  Taking $D=0$ gives the
$K_X$-version, while taking $D=K_Y$ gives the relative canonical formulation stated above.

If $\varphi\colon X\to Y$ is a projective space bundle, then $K_{X/Y}$ is always Cartier. On the other hand, 
\(K_X\) need not be \(\mathbb Q\)-Cartier when
\(K_Y\) is not \(\mathbb Q\)-Cartier. 
Therefore, it seems natural to use $K_{X/Y}$ to formulate Theorem \ref{thm:affine-local}.
\end{rem}

\begin{rem}
In Theorem~\ref{thm:affine-local}, since \(Y\) is affine, every complete curve in \(X\) is contracted by
\(\varphi\). Hence \eqref{eq:curve-inequality} and the toric relative
ampleness criterion imply that \(-K_{X/Y}\) is \(\varphi\)-ample.
Therefore the proper morphism \(\varphi\) is projective. 
\end{rem}

\subsection{The general fiber}

\begin{lem}
\label{lem:general-fiber}
Under the assumptions of Theorem~\ref{thm:affine-local}, the general fiber of $\varphi$
is isomorphic to $\PP^d$.
\end{lem}

\begin{proof}
Let $T_Y\subset Y$ be the dense torus and let $F$ be a fiber over a point of $T_Y$.
After choosing a splitting of the exact sequence of lattices over the dense torus, the
restriction of $\varphi$ over $T_Y$ is a product with fiber $F$.  Since the standard
torus-invariant canonical divisor $K_Y$ restricts to zero on $T_Y$, we obtain
\[
  K_{X/Y}|_F=K_F.
\]
Every curve $C\subset F$ is complete in $X$, and hence
\[
  -K_F\cdot C=-K_{X/Y}\cdot C>d.
\]
Lemma~\ref{ext:fujino_toric_cone} now implies that $F\simeq\PP^d$.
\end{proof}

Since \(\varphi\) has connected fibers, 
\(p\colon N\to N_Y\) is surjective. 
Put \(N_F:=\ker(p)\). By Lemma~\ref{lem:general-fiber}, the fan in $(N_F)_{\RR}$ is the standard fan of
$\PP^d$.  We choose its primitive ray generators
\[
  v_0,v_1,\ldots,v_d\in N_F
\]
so that $v_1,\ldots,v_d$ is a basis of $N_F$ and
\begin{equation}
  v_0+v_1+\cdots+v_d=0.
  \label{eq:Pd-relation}
\end{equation}
For $0\leq i\leq d$, put
\[
  \gamma_i:=\Cone(v_0,\ldots,\widehat v_i,\ldots,v_d).
\]

\subsection{Unique primitive lifts of the rays of the base}

We now return to Theorem~\ref{thm:affine-local}.  By separating a torus factor, we may
assume that
\[
  Y=U_\tau
\]
for a full-dimensional strongly convex cone $\tau\subset (N_Y)_{\RR}$.

\begin{lem}[Unique primitive lift]
\label{lem:unique-lifts}
For every ray $\alpha\in\Delta(1)$, there is exactly one ray
$\rho_\alpha\in\Sigma(1)$ whose image is contained in $\alpha$ and is nonzero.  If
$w_\alpha$ and $u_{\rho_\alpha}$ are the primitive generators, then
\[
  p(u_{\rho_\alpha})=w_\alpha.
\]
The rays contained in $\ker(p)_{\RR}$ are exactly
$\RR_{\geq0}v_0,\ldots,\RR_{\geq0}v_d$.
\end{lem}

\begin{proof}
Fix a ray $\alpha\in\Delta(1)$, with primitive generator $w_\alpha$.  The toric open
subset $U_\alpha\subset Y$ is noncanonically isomorphic to
\[
  \mathbb A^1\times(\GGm)^{\dim Y-1}.
\]
Put
\[
  N_\alpha:=p^{-1}(\ZZ w_\alpha).
\]
The part of the fan of $X$ lying over $U_\alpha$ is contained in
$(N_\alpha)_{\RR}$.  After choosing a complementary lattice, we obtain a product
decomposition
\[
  \varphi^{-1}(U_\alpha)
  \simeq V_\alpha\times(\GGm)^{\dim Y-1},
\]
where
\[
  f_\alpha:V_\alpha\longrightarrow\mathbb A^1
\]
is a projective surjective equidimensional toric morphism with connected fibers and
relative dimension $d$.

The variety $U_\alpha$ is smooth.  Hence $K_Y|_{U_\alpha}$ is Cartier, and on
$\varphi^{-1}(U_\alpha)$ we have
\[
  K_X=K_{X/Y}+\varphi^*(K_Y|_{U_\alpha}).
\]
It follows that $K_X$ is $\mathbb Q$-Cartier on this open set, and therefore
$V_\alpha$ is $\mathbb Q$-Gorenstein.  If $C\subset V_\alpha$ is a complete curve,
then $C\times\{1\}$ is a complete curve in $X$.  Since it is vertical over
$U_\alpha$, the Cartier divisor $\varphi^*(K_Y|_{U_\alpha})$ has degree zero on it.
Consequently,
\[
  -K_{V_\alpha}\cdot C
  =-K_X\cdot(C\times\{1\})
  =-K_{X/Y}\cdot(C\times\{1\})
  >d.
\]
Theorem~\ref{thm:A1-version} gives
\[
  V_\alpha\simeq\PP^d\times\mathbb A^1.
\]
The fan of this product contains exactly the $d+1$ vertical rays and one horizontal ray,
and the primitive generator of the horizontal ray maps to $w_\alpha$.  This proves the
claim for every $\alpha$.  The assertion about the vertical rays follows from the fan of
the general fiber.
\end{proof}

\subsection{The weak splitting of the fan}

The next lemma shows that, once the rays have been controlled, the real cone structure is
forced by equidimensionality and properness.

\begin{lem}[Weak splitting]
\label{lem:weak-splitting}
For every face $\sigma\preceq\tau$, set
\[
  \widehat\sigma
  :=\Cone\bigl(u_{\rho_\alpha}\mid \alpha\in\sigma(1)\bigr).
\]
Then $\widehat\sigma$ is a cone of $\Sigma$, and
\[
  p_{\RR}|_{\widehat\sigma}:\widehat\sigma\xrightarrow{\sim}\sigma
\]
is an isomorphism of real cones.  The cones $\widehat\sigma$ form a subfan
$\widehat\Delta\subset\Sigma$.

Furthermore,
\begin{equation}
  \Sigma
  =
  \left\{
    \widehat\sigma+\gamma
    \ \middle|\ 
    \sigma\preceq\tau,\ \gamma\in\Sigma_F
  \right\},
  \label{eq:weak-splitting}
\end{equation}
where $\Sigma_F$ is the standard fan of $\PP^d$.
\end{lem}

\begin{proof}
Fix a face $\sigma\preceq\tau$.  By Proposition~\ref{prop:AK}, the image of every cone
of $\Sigma$ is a cone of $\Delta$.

Consider the fiber over a point of the orbit $O(\sigma)$.  If a cone
$\kappa\in\Sigma$ maps onto $\sigma$, then the corresponding orbit has relative
dimension
\begin{equation}
  \bigl(\dim X-\dim\kappa\bigr)
  -\bigl(\dim Y-\dim\sigma\bigr)
  =d+\dim\sigma-\dim\kappa.
  \label{eq:orbit-dimension}
\end{equation}
Equidimensionality implies that this number is at most $d$, so
$\dim\kappa\geq\dim\sigma$.  Since the fiber has dimension exactly $d$, some stratum
of a $d$-dimensional irreducible component gives a cone $\kappa$ for which equality
holds.  Thus
\[
  p_{\RR}(\kappa)=\sigma,
  \qquad
  \dim\kappa=\dim\sigma.
\]
The restriction of $p_{\RR}$ to the span of $\kappa$ is therefore an isomorphism.
Hence the extremal rays of $\kappa$ map bijectively to the extremal rays of $\sigma$.
By Lemma~\ref{lem:unique-lifts}, each ray of $\sigma$ has only one possible lift.
Consequently,
\[
  \kappa=\widehat\sigma.
\]
This proves that $\widehat\sigma\in\Sigma$ and that it maps isomorphically to $\sigma$.
The same cone is forced for every face, so
\[
  \sigma'\preceq\sigma
  \quad\Longrightarrow\quad
  \widehat{\sigma'}\preceq\widehat\sigma.
\]
Thus the $\widehat\sigma$ form a subfan.

We now consider the unique maximal cone $\tau$ of the affine base.  Put
$H:=\Span_{\RR}(\widehat\tau)$.  Since $p_{\RR}|_H$ is an isomorphism, we have a
direct-sum decomposition of real vector spaces
\[
  N_{\RR}=H\oplus (N_F)_{\RR}.
\]
The orbit closure $V(\widehat\tau)$ lies in the proper fiber over the closed orbit of
$Y$, so its fan
\[
  \operatorname{Star}_\Sigma(\widehat\tau)
\]
in $N_{\RR}/H$ is complete.  Every ray of this star is represented by a ray of $\Sigma$
which is not contained in $\widehat\tau$.  By Lemma~\ref{lem:unique-lifts}, all
nonvertical rays over $\tau$ already belong to $\widehat\tau$. 
Since the star is a complete \(d\)-dimensional fan, 
all of the \(d+1\) vertical rays must occur. 
Hence the rays of the
star are precisely the vertical rays
\[
  \RR_{\geq0}v_0,\ldots,\RR_{\geq0}v_d.
\]
There is only one complete fan with these $d+1$ rays, namely the standard fan of
$\PP^d$.  Therefore
\[
  \operatorname{Star}_\Sigma(\widehat\tau)=\Sigma_F.
\]
It follows that the maximal cones containing $\widehat\tau$ are exactly
\[
  \widehat\tau+\gamma_i,
  \qquad 0\leq i\leq d.
\]
Their union is
\[
  \widehat\tau+|(\Sigma_F)|
  =\widehat\tau+(N_F)_{\RR}
  =p_{\RR}^{-1}(\tau),
\]
which is the support of the fan over $Y$ by properness.

Finally, every cone of $\Sigma$ is contained in one of these maximal product cones.
Indeed, the product cones themselves are cones of $\Sigma$ and cover the support; the
relative interior of a cone of $\Sigma$ cannot be covered by finitely many proper faces.
Thus every cone is a face of some $\widehat\tau+\gamma_i$.  Because the sum is direct in
real vector spaces, every such face is of the form $\widehat\sigma+\gamma$, with
$\sigma\preceq\tau$ and $\gamma\in\Sigma_F$.  This proves
\eqref{eq:weak-splitting}.
\end{proof}

\subsection{Eliminating the lattice defect}

Although the fan is now a product at the level of real cones, the horizontal lattice may
still have finite index in the lattice of the base.  The relative anticanonical
inequality eliminates this last possibility.

\begin{lem}[Saturation of the horizontal lattice]
\label{lem:saturation}
For every face $\sigma\preceq\tau$, one has
\begin{equation}
  p(\widehat\sigma\cap N)=\sigma\cap N_Y.
  \label{eq:saturation}
\end{equation} 
In particular, \(\Sigma\) is split by
\(\widehat{\Delta}\) and \(\Sigma_F\) in the sense of
\cite[Definition~3.3.18]{cox-little-toric}. 
\end{lem}

\begin{proof}
It is enough to prove the assertion for the maximal cone $\tau$; the statement for its
faces then follows by restriction.

Set
\[
  H:=\Span_{\RR}(\widehat\tau),
  \qquad
  L:=N\cap H.
\]
The map $p|_L$ is injective, and $p(L)$ is a finite-index sublattice of $N_Y$.  Suppose,
for a contradiction, that
\[
  p(L)\subsetneq N_Y.
\]
Since $H$ is rational, $L$ is saturated in $N$, and hence
\[
  \overline N:=N/L
\]
is a lattice of rank $d$.  Moreover, $N_F\cap L=0$.  Surjectivity of $p$ gives an exact
sequence
\begin{equation}
  0\longrightarrow N_F
  \longrightarrow \overline N
  \longrightarrow N_Y/p(L)
  \longrightarrow0.
  \label{eq:overlattice}
\end{equation}
Thus $N_F$ is a proper finite-index sublattice of $\overline N$.

For $1\leq i\leq d$, put
\[
  \delta_i:=\Cone(v_1,\ldots,\widehat v_i,\ldots,v_d)
\]
and consider the rank-one lattice
\[
  \overline N_i
  :=
  \overline N\Big/
  \bigl(\overline N\cap\Span_{\RR}(\delta_i)\bigr).
\]
Let $e_i$ be its primitive generator in the direction of $v_i$, and write
\[
  \overline v_i=q_i e_i,
  \qquad q_i\geq1.
\]
By \eqref{eq:Pd-relation}, the image of $v_0$ is $-q_i e_i$.

We claim that $q_i\geq2$ for at least one $i$.  Let
$\varepsilon_1,\ldots,\varepsilon_d$ be the basis of
$\operatorname{Hom}(N_F,\ZZ)$ dual to $v_1,\ldots,v_d$.  The equality $q_i=1$ is
equivalent to $\varepsilon_i$ taking integral values on $\overline N$.  If $q_i=1$ for
every $i$, then every $x\in\overline N$ has an expression
\[
  x=\sum_{i=1}^d\varepsilon_i(x)v_i
\]
with integral coefficients.  Hence $x\in N_F$, which contradicts
$N_F\subsetneq\overline N$.  The claim follows.

Fix $i$ with $q_i\geq2$ and put
\[
  \eta_i:=\widehat\tau+\delta_i.
\]
By Lemma~\ref{lem:weak-splitting}, this is a cone of $\Sigma$ of codimension one.
Thus
\[
  C_i:=V(\eta_i)
\]
is a complete torus-invariant curve in the fiber over the closed orbit of $Y$.  The two
maximal cones adjacent to $\eta_i$ are
\[
  \widehat\tau+\gamma_0
  =\eta_i+\RR_{\geq0}v_i
\]
and
\[
  \widehat\tau+\gamma_i
  =\eta_i+\RR_{\geq0}v_0.
\]

Let $m_+$ and $m_-$ be the local rational linear functions defining $-K_{X/Y}$ on
these two cones. 
We use the convention that the local linear function associated to a
torus-invariant divisor takes the
coefficient of the corresponding invariant prime divisor on each primitive ray generator. 
Since $K_{X/Y}$ is $\mathbb Q$-Cartier, such functions exist and
agree on $\eta_i$.  The divisor $\varphi^{[*]}K_Y$ has coefficient zero along every
vertical divisor.  Hence $-K_{X/Y}$ has coefficient $1$ along every vertical ray, just
as $-K_X$ does.  Therefore
\[
  m_+(v_i)=1.
\]
The second cone contains
$v_0,v_1,\ldots,\widehat v_i,\ldots,v_d$, and $m_-$ takes the value $1$ on all of
these rays.  From \eqref{eq:Pd-relation},
\[
  v_i=-\left(v_0+\sum_{j\neq i}v_j\right),
\]
so
\[
  m_-(v_i)=-d.
\]
Hence
\[
  (m_+-m_-)(v_i)=d+1.
\]

The quotient lattice attached to the invariant curve $C_i$ is naturally
$\overline N_i$, and the image of $v_i$ is $q_i e_i$.  The standard toric slope formula
for a $\mathbb Q$-Cartier divisor on an invariant curve gives
\[
  -K_{X/Y}\cdot C_i
  =(m_+-m_-)(e_i)
  =\frac{d+1}{q_i}
  \leq\frac{d+1}{2}
  \leq d.
\]
This contradicts \eqref{eq:curve-inequality}.  Therefore $p(L)=N_Y$.
Since $p_{\RR}$ maps $\widehat\tau$ isomorphically onto $\tau$, we obtain
\[
  p(\widehat\tau\cap N)=\tau\cap N_Y.
\]
The assertion for every face follows by restriction.
\end{proof}

\subsection{Proof of the affine local theorem}

\begin{proof}[Proof of Theorem~\ref{thm:affine-local}]
After removing a torus factor, we may assume that the cone defining the affine toric
variety $Y$ is full-dimensional.  Lemma~\ref{lem:general-fiber} identifies the fiber fan
with the standard fan $\Sigma_F$ of $\PP^d$.  Lemma~\ref{lem:unique-lifts} gives a
unique primitive lift of every ray of the base.  Lemma~\ref{lem:weak-splitting} gives a
horizontal subfan $\widehat\Delta$ and the weak product decomposition
\eqref{eq:weak-splitting}.  Lemma~\ref{lem:saturation} gives the integral lattice equality
\eqref{eq:saturation}.

By Lemmas~\ref{lem:weak-splitting} and~\ref{lem:saturation},
$\Sigma$ is split by
$\widehat{\Delta}$ and
$\Sigma_F$.
Hence
\cite[Theorem~3.3.19]{cox-little-toric}
yields a locally trivial toric fiber bundle with fiber
$\mathbb{P}^d$. 

Since \(Y=U_\tau\) consists of a single affine toric chart, the local
trivialization furnished by
\cite[Theorem~3.3.19]{cox-little-toric}
is global. Consequently,
\[
X\simeq\mathbb P^d\times Y
\]
equivariantly over \(Y\). 
\end{proof}

\subsection{Application to long relative extremal contractions}

We record the global consequence for which the local theorem was designed.

\begin{cor}[Global relative projective-space-bundle statement]
\label{cor:global}
Let $X$ be a projective toric variety, and let
\[
  \varphi_R:X\longrightarrow W
\]
be a projective toric contraction associated to an extremal ray $R\subset\NE(X)$.
Assume that
\[
  d:=\dim X-\dim W>0,
\]
that $\varphi_R$ is equidimensional, and that the toric relative canonical divisor
\[
  K_{X/W}:=K_X-\varphi_R^{[*]}K_W
\]
is $\mathbb Q$-Cartier.  Define the relative length of $R$ by
\[
  l_{X/W}(R)
  :=
  \min_{[C]\in R}\bigl(-K_{X/W}\cdot C\bigr),
\]
where $C$ runs through complete integral curves whose numerical classes belong to $R$.
If
\[
  l_{X/W}(R)>d,
\]
then $\varphi_R$ is a toric $\PP^d$-bundle.
\end{cor}

\begin{proof}
Let $U\subset W$ be a maximal affine toric chart and put
$X_U:=\varphi_R^{-1}(U)$.  The restricted morphism
\[
  X_U\longrightarrow U
\]
is proper, surjective, equidimensional, toric, and has connected fibers.  Moreover, its
toric relative canonical divisor is the restriction of $K_{X/W}$ and is therefore
$\mathbb Q$-Cartier.

Let $C\subset X_U$ be a complete curve.  Since $U$ is affine, the image of $C$ is a
point.  Thus $C$ is contracted by $\varphi_R$, and its numerical class belongs to $R$.
Hence
\[
  -K_{X_U/U}\cdot C
  =-K_{X/W}\cdot C
  \geq l_{X/W}(R)>d.
\]
Therefore Theorem~\ref{thm:affine-local} applies and gives
\[
  X_U\simeq\PP^d\times U
\]
over $U$.  The maximal affine toric charts cover $W$, so these local products show that
$\varphi_R$ is a Zariski locally trivial toric $\PP^d$-bundle.
\end{proof}

\begin{rem}
Under the assumptions of Corollary~\ref{cor:global}, replace the hypotheses that \(K_{X/W}\) is \(\mathbb{Q}\)-Cartier and that \(l_{X/W}(R)>d\) by the assumptions that \(K_X\) is \(\mathbb{Q}\)-Cartier and that
\[
l(R):=
\min_{[C]\in R}(-K_X\cdot C)>d,
\]
where \(C\) runs through complete integral curves. Then the same conclusion holds: \(\varphi_R\) is a toric \(\mathbb{P}^d\)-bundle.

Indeed, let \(U\subset W\) be a maximal affine toric chart and put \(X_U:=\varphi_R^{-1}(U)\). Every complete curve \(C\subset X_U\) is contracted by \(\varphi_R\), and hence
\[
-K_X\cdot C\geq l(R)>d.
\]
Therefore the \(K_X\)-version of Theorem~\ref{thm:affine-local} given in Remark~\ref{rem:general-divisor} applies to \(X_U\to U\). The same argument as in the proof of Corollary~\ref{cor:global} shows that \(\varphi_R\) is a toric \(\mathbb{P}^d\)-bundle. 

This may be viewed as an equidimensional non-\(\mathbb{Q}\)-factorial
counterpart of \cite[Corollary~3.3]{fujino-length}. 
\end{rem}

\begin{rem}
Under the assumptions of Corollary~\ref{cor:global}, if $K_W$ is
$\mathbb{Q}$-Cartier, then
\[
K_X=K_{X/W}+\varphi_R^*K_W
\]
is also $\mathbb{Q}$-Cartier. Moreover,
\[
-K_{X/W}\cdot C=-K_X\cdot C
\]
for every curve contracted by $\varphi_R$. Thus
\[
l_{X/W}(R)=l(R),
\]
and the relative and usual formulations agree. 
\end{rem}

\begin{rem}
The affine theorem does not require a relative Picard-number hypothesis.  In the global
corollary, extremality is used only to guarantee that every complete curve in the inverse
image of an affine chart belongs to the single ray $R$ and hence satisfies the required
relative anticanonical inequality.
\end{rem}

The following example shows that the equidimensionality assumption in
Corollary~\ref{cor:global} cannot be dropped.

\begin{ex}
Let
\[
  N=\mathbb Z^3,
  \qquad
  p\colon N\longrightarrow N_W=\mathbb Z^2,
  \qquad
  p(x,y,z)=(x,y).
\]
The toric variety $\mathbb P^1\times\mathbb P^2$ has fan in $N$ with rays generated by
\[
\begin{aligned}
  h_1&=(1,0,0), & h_2&=(0,1,0), & h_3&=(-1,-1,0),\\
  v_+&=(0,0,1), & v_-&=(0,0,-1). &
\end{aligned}
\]
The second projection $\mathbb P^1\times\mathbb P^2\to W=\mathbb P^2$ is associated
to $p$.  The rays of the fan of $W$ are generated by
\[
  \bar h_1=(1,0),\qquad
  \bar h_2=(0,1),\qquad
  \bar h_3=(-1,-1).
\]
Let
\[
  \mu\colon\widetilde X\longrightarrow\mathbb P^1\times\mathbb P^2
\]
be the blow-up at the point $V(\langle h_1,h_2,v_+\rangle)$.  Thus the fan of
$\widetilde X$ has one additional ray generated by
\[
  e:=h_1+h_2+v_+.
\]
Its primitive relations are
\[
  h_1+h_2+v_+=e,\qquad h_3+e=v_+,\qquad v_-+e=h_1+h_2,
\]
and
\[
  h_1+h_2+h_3=0,\qquad v_++v_-=0.
\]
The first three relations are extremal.  Let
\[
  \psi\colon\widetilde X\longrightarrow X=X_\Sigma
\]
be the contraction associated to the third relation. 
The contraction $\psi$ is small, 
and $X$ is a non-$\mathbb{Q}$-factorial 
projective toric threefold. 
A direct check of the canonical support 
function shows that $X$ is Gorenstein and 
that $\psi$ is crepant. 
There is a
toric morphism
\[
  \varphi_R\colon X\longrightarrow W
\]
with $\rho(X/W)=1$, associated to an extremal ray $R\subset\NE(X)$.  Its general fiber
is $\mathbb P^1$, so $d=1$.  On the other hand, the divisor
$V(\langle e\rangle)\subset X$ is contracted to the point
$V(\langle\bar h_1,\bar h_2\rangle)\in W$.  Thus $\varphi_R$ is not equidimensional.

The torus-invariant curves contracted by $\varphi_R$ are
\[
  C_1:=V(\langle h_1,e\rangle),\quad
  C_2:=V(\langle h_2,e\rangle),\quad
  C_3:=V(\langle v_+,e\rangle),
\]
and
\[
  C_4:=V(\langle h_1,h_3\rangle),\qquad
  C_5:=V(\langle h_2,h_3\rangle).
\]
We have
\[
  (-K_X\cdot C_4)=(-K_X\cdot C_5)=2.
\]
On $\widetilde X$, put
\[
  C'_1:=V(\langle h_1,e\rangle),\quad
  C'_2:=V(\langle h_2,e\rangle),\quad
  C'_3:=V(\langle v_+,e\rangle).
\]
The wall relations give
\[
  (-K_{\widetilde X}\cdot C'_1)
  =(-K_{\widetilde X}\cdot C'_2)
  =(-K_{\widetilde X}\cdot C'_3)=2.
\]
Since $\psi$ is small and crepant,
\[
  (-K_X\cdot C_1)=(-K_X\cdot C_2)=(-K_X\cdot C_3)=2.
\]
Because $W=\mathbb P^2$ is smooth, $K_{X/W}=K_X-\varphi_R^*K_W$ and
\[
  -K_{X/W}\cdot C_i=-K_X\cdot C_i=2
  \qquad(1\leq i\leq5).
\] Since every effective curve 
on a complete toric variety is numerically 
equivalent to an effective torus-invariant \(1\)-cycle, 
the above computations show that 
\[
  l_{X/W}(R)=2>1=d,
\]
although $\varphi_R\colon X\to W$ is not a $\mathbb P^1$-bundle.
\end{ex}

\bibliographystyle{amsalpha}
\bibliography{ref}

\end{document}